\documentclass[10pt]{article}

\usepackage[margin=1in]{geometry} 
\usepackage{amsmath,amsthm,amssymb,amsfonts}
\usepackage{graphicx, multicol, array}
\usepackage{cases}
\usepackage{enumerate}
\usepackage{enumitem}
\usepackage{hyperref}

\usepackage{mathrsfs}

\DeclareMathOperator{\li}{li}

\newtheorem{thm}{Theorem}[section]

\newtheorem{conj}{Conjecture}[section]

\title{\textbf{Inequalities For The Primes Counting Function}}
\date{}
\author{N. A. Carella}

\usepackage{fancyhdr}
\begin{document}
\thispagestyle{empty}
\maketitle

\vskip .25 in 
\textbf{Abstract:} The prime counting function inequality $\pi(x+y) < \pi(x)+\pi(y)$, which is known as Hardy-Littlewood conjecture, has been established for a variety of cases such as $ \delta x \leq y \leq x$, where $0< \delta \leq 1$, and $x \leq y\leq x \log x \log \log x$ as $ x \to \infty$. The goal in note is to extend the inequality to the new larger ranges $\geq x \log^{-c}x\leq y \leq x$, where $c\geq 0$ is a constant, unconditionally; and for $\geq x^{1/2} \log^3x\leq y \leq x$, conditional on a standard conjecture. 
\\


\vskip .25 in 

\textbf{Keyword:}
Distribution of prime; Prime in short interval; Hardy-Littlewood conjecture; Prime $k$-tuple.\\

 \textbf{AMS Mathematical Subjects Classification:} 11N05, 11A41, 11M26.\\

\vskip .75 in

\section{Introduction}
Let $ x\geq 1$ be a large number, and let $\pi(x)= \#\{p\leq x: p\text{ is prime} \}$. There are many partial results for the prime counting function inequality 
\begin{equation} \label{eq120.99}
\pi(x+y) < \pi(x)+\pi(y).
\end{equation}
The range of parameter $\delta x \leq y \leq x$, with a constant $0<\delta \leq 1$ as $x \to \infty$, is proved in \cite[Theorem 3]{PL00}. The range of parameter $x \leq y\leq x \log x \log \log x$ as $x \to \infty$, is proved in \cite{DP98}. Various other related inequalities are proved in \cite{MV75}, \cite{GR02}, \cite{PL02}, and by other authors. The Hardy-Littlewood conjecture states that inequality (\ref{eq120.99}) is valid for any $x \geq 2$ and any $y \geq 2$. The goal in note is to extend the inequality to larger ranges than it is currently known. It is shown that it is valid for the range of parameter $x \log^{-c}x\leq y \leq x$, where $c\geq 0$ is a constant, unconditionally. And for $ x^{1/2} \log^3x \leq y \leq x$, conditional on the RH. These are new results in the mathematical literature.

\begin{thm} \label{thm285.01} Let $x\geq 2$ be a large number, and let $x \log^{-c}x\leq y \leq x$, with $c\geq 0$ an arbitrary constant. Then
\begin{equation}
 \pi(x+y) < \pi(x)+\pi(y).
\end{equation}
\end{thm}

\begin{thm} \label{thm295.02} Let $x\geq 2$ be a large number, and let $x^{1/2} \log^3x \leq y \leq x$. Assume the nontrivial zeros of the zeta function are on the line $\Re e(s)=1/2$. Then
\begin{equation}
 \pi(x+y) < \pi(x)+\pi(y).
\end{equation}
\end{thm}

The proof of Theorem \ref{thm285.01} is assembled in Section \ref{s285}, and the proof of Theorem \ref{thm295.02} is assembled in Section \ref{s295}. The penultimate section inquires on the limitiation of the prime counting inequality (\ref{eq120.99}).
\section{Unconditional Result} \label{s285}
The unconditional result for the \textit{prime number theorem}, see Theorem \ref{thm200.02}, together with the \textit{mean value theorem} for integral, see Theorem \ref{thm555.01}, give a nice and simple proof for the prime counting function inequality (\ref{eq120.99}) over the larger interval $[x, x+y]$ with $x \log^{-c}x\leq y \leq x$, for any constant $c\geq 0$.
\begin{proof} (Theorem  \ref{thm285.01}) By the \textit{prime number theorem}, the prime counting function has the integral representation
\begin{equation} \label{eq285.33}
\pi(x)=\int_2^x\frac{1}{\log t} dt+O\left (x e^{-\sqrt{\log x}} \right ) ,
\end{equation}
see Theorem \ref{thm200.02}. Accordingly, the reverse inequality $\pi(x+y) \geq \pi(x)+\pi(y)$ has the integral representation
\begin{equation} \label{eq285.35}
\int_x^{x+y}\frac{1}{\log t} dt+O\left (x e^{-\sqrt{\log x}} \right )\geq \int_2^{y}\frac{1}{\log t} dt +O\left (x e^{-\sqrt{\log x}} \right ).
\end{equation}
By the \textit{mean value theorem} for integral, there is a value $x_0 \in (x,x+y)$, and a value $x_1 \in (2,y)$ such that
\begin{eqnarray} \label{eq285.375}
\frac{y}{\log x_0}+O\left (x e^{-\sqrt{\log x}} \right )&=&\int_x^{x+y}\frac{1}{\log t} dt+O\left (x e^{-\sqrt{\log x}} \right )\nonumber \\ 
&\geq& \int_2^{y}\frac{1}{\log t} dt+O\left (x e^{-\sqrt{\log x}} \right )) \\ 
&=&\frac{y-2}{\log x_1} +O\left (x e^{-\sqrt{\log x}} \right )\nonumber.
\end{eqnarray}
Dividing the left and right sides of (\ref{eq285.375}) by  $y \geq x \log^{-c} x$, and multiplying it by $\log x_1$ give
\begin{equation} \label{eq285.379}
\frac{\log x_1}{\log x_0}+O\left (\frac{(\log x_1)\log^{c}x}{ e^{\sqrt{\log x}}} \right )\geq 1-\frac{2}{y}.
\end{equation}
Based on the data 
\begin{equation} \label{eq285.325}
x \leq x_0 \leq  x+y=x+x \log^{-c} x \qquad \text{   and   } \qquad 2 \leq x_1 \leq y=x \log^{-c} x,
\end{equation}
the upper bound of the left side of (\ref{eq285.379}) is
\begin{eqnarray} \label{eq285.342}
\frac{\log x_1}{\log x_0}+O\left (\frac{(\log x_1)\log^{c}x}{ e^{\sqrt{\log x}}} \right )
&\leq &\frac{\log \left( x \log^{-c} x \right ) }{\log x}+O\left (\frac{\log^{c+1}x}{ e^{\sqrt{\log x}}} \right ) \nonumber \\ 
&\leq& 1-\frac{c\log \log  x}{\log x}+O\left (\frac{\log^{c+1}x}{ e^{\sqrt{\log x}}} \right ).
\end{eqnarray}
And the upper bound of the right side of (\ref{eq285.379}) is
\begin{equation} \label{eq285.345}
1-\frac{2}{y}\leq 1-\frac{2\log^{c}x}{ x}.
\end{equation}
Replacing (\ref{eq285.342}) and (\ref{eq285.345}) into (\ref{eq285.379}) yield
\begin{equation} \label{eq285.347}
1-\frac{2\log^{c}x}{ x} \leq 1-\frac{c\log \log  x}{\log x}+O\left (\frac{\log^{c+1}x}{ e^{\sqrt{\log x}}} \right ).
\end{equation}
Since the left side increases at a faster rate than the right side, this is a contradiction as $x \to \infty$. Ergo, $\pi(x+y) < \pi(x)+\pi(y)$.
\end{proof}

\section{Conditional Result} \label{s295}
 The conditional results for the zeta function and the \textit{prime number theorem}, see Theorem \ref{thm200.02}, together with the \textit{mean value theorem} for integral, are sufficient to extend the prime counting function inequality (\ref{eq120.99}) to the larger interval $[x,x+y]$ with $ x^{1/2} \log^3 x \leq y \leq x$. This is a new result in the mathematical literature.

\begin{proof} (Theorem \ref{thm295.02}) The prime counting function has the integral representation
\begin{equation} \label{eq295.33}
\pi(x)=\int_2^x\frac{1}{\log t} dt+O\left (x^{1/2} \log x \right ) ,
\end{equation}
see Theorem \ref{thm200.02}. Accordingly, the reverse inequality $\pi(x+y) \geq \pi(x)+\pi(y)$ has the integral representation
\begin{equation} \label{eq295.35}
\int_x^{x+y}\frac{1}{\log t} dt+O\left (x^{1/2} \log x \right )\geq \int_2^{y}\frac{1}{\log t} dt+O\left (x^{1/2} \log x \right ).
\end{equation}
By the mean value theorem for integral, there is a value $x_0 \in (x,x+y)$, and a value $x_1 \in (2,y)$ such that
\begin{eqnarray} \label{eq295.375}
\frac{y}{\log x_0}+O\left (x^{1/2} \log x \right )&=&\int_x^{x+y}\frac{1}{\log t} dt+O\left (x^{1/2} \log x \right )\nonumber\\
&\geq& \int_2^{y}\frac{1}{\log t} dt+O\left (x^{1/2} \log x \right )\\
&=&\frac{y-2}{\log x_1} +O\left (x^{1/2} \log x \right )\nonumber.
\end{eqnarray}
Dividing the left and right sides of (\ref{eq295.375}) by $y\geq x^{1/2} \log^3 x$, and multiplying it by $\log x_1$ give
\begin{equation} \label{eq295.379}
\frac{\log x_1}{\log x_0}+O\left (\frac{\log x_1}{ \log^2 x} \right )\geq 1-\frac{2}{y}.
\end{equation}
Based on the data 
\begin{equation} \label{eq295.325}
x \leq x_0 \leq  x+y=x+x^{1/2} \log^3 x \qquad \text{   and   } \qquad 2 \leq x_1 \leq y=x^{1/2} \log^3 x,
\end{equation}
the upper bound of the left side of (\ref{eq295.379}) is
\begin{eqnarray} \label{eq295.342}
\frac{\log x_1}{\log x_0}+O\left (\frac{\log x_1}{ \log^2 x} \right )
&\leq &\frac{\log \left(x^{1/2} \log^3 x \right ) }{\log x}+O\left (\frac{1}{ \log x} \right ) \\
&\leq& \frac{1}{2}+\frac{\log \log^3 x}{\log x}+O\left (\frac{1}{ \log x} \right ) \nonumber.
\end{eqnarray}
And the upper bound of the right side of (\ref{eq295.379}) is
\begin{equation} \label{eq295.345}
1-\frac{2}{y}\leq 1-\frac{2}{x^{1/2}\log^3x}.
\end{equation}
Replacing (\ref{eq295.342}) and (\ref{eq295.345}) into (\ref{eq295.379}), yield
\begin{equation} \label{eq295.347}
1-\frac{2}{x^{1/2}\log^3x}\leq \frac{1}{2}+\frac{\log \log^3 x}{\log x}+O\left (\frac{1}{ \log x} \right ).
\end{equation}
Trivially, this is a contradiction for all large numbers $x \geq 2$. Ergo, $\pi(x+y) < \pi(x)+\pi(y)$.
\end{proof}

\section{Limits Of the Primes Counting Function Inequality} \label{s275}
 There are several results for the oscillations of the primes counting function over small intervals, confer Theorem \ref{thm200.02} and Theorem  \ref{thm200.58}. The oscillations of the values of the prime counting function seems to force some limits on the conjectured Hardy-Littlewood inequality
\begin{equation}
\pi(x+y) < \pi(x)+\pi(y).
\end{equation}

\begin{conj} \label{thm275.01} Let $x\geq 2$ be a large number. Let $y=\log^r x$, where $r >0$ a real number. Then
\begin{equation}
 \pi(x+\log^r x) < \pi(x)+\pi(\log^r x) 
\end{equation}
and 
\begin{equation}
 \pi(x+\log^r x) > \pi(x)+\pi(\log^r x) 
\end{equation}
infinitely often as $x \to \infty$.
\end{conj}
It is not clear if it can be proved or disproved by elementary methods, for example, using Theorem \ref{thm200.58} or Theorem  \ref{thm200.02}. In synopsis, this inequality is likely to fail on very short intervals. This has some relevance to the prime $k$-tuples conjecture, see \cite{HR73}. Some detailed information on the hierarchy of  prime $k$-tuples conjectures are explicated in a new survey, \cite[p.\ 10]{FG18}.

\section{Prime Numbers Theorems} \label{s200}
The omega notation $f (x)= g(x)+\Omega_{\pm}(h(x))$ means that both $f (x) > g(x)+ c_0h(x)$ and $f (x) < g(x)- c_1h(x)$ occur infinitely often as $x \to \infty$, where $c_0 > 0$ and $c_1 > 0$ are constants, see \cite[p.\ 5]{MV07}, and similar references. The set of prime numbers is denoted by $\mathbb{P}=\{2,3,5, \ldots \}$, and for a real number $x \geq 1$, the standard prime counting function is denoted by 
\begin{equation}
\pi(x)= \#\{ p \leq x : p \text{ prime }\}=\sum_{p \leq x} 1.   
\end{equation}
In addition, the logarithm integral and multiple logarithm integral are defined by $\li(x)=\int_2^x \frac{1}{\log t}dt$ and $\li_k(x)=\int_2^x \frac{1}{\log^k t}dt$, $k \geq 1$, respectively. The weighted primes counting functions, psi $\psi(x)$ and theta $\theta(x)$, are defined by
\begin{equation}
\theta(x)=\sum_{p \leq x} \log p \quad \text{ and } \quad \psi(x)=\sum_{p^k \leq x} \log p^k                 
\end{equation}
respectively.

\begin{thm} \label{thm200.01} Uniformly for $x \geq 2$ the psi and theta functions have the followings asymptotic formulae.
\begin{enumerate} [font=\normalfont, label=(\roman*)]
\item  \text{Unconditionally,}
$$
\theta(x)=x +O\left (xe^{-c_0 \sqrt{\log x}}\right ) .
$$
\item \text{Unconditional oscillation,}
$$
\theta(x)=x +\Omega_{\pm} \left ( x^{1/2}\log \log \log x \right ).  $$

\item \text{Conditional on the RH,}
$$
\theta(x)=x +O\left (x^{1/2} \log^2 x \right ).  
$$
\end{enumerate}
\end{thm}

\begin{proof} (ii) The oscillations form of the theta function is proved in \cite[p.\ 479]{MV07}, 
\end{proof}

The same asymptotics hold for the function $\psi(x)$. Explicit estimates for both of these functions are given in \cite{CP85}, \cite{SL76}, \cite[Theorem 5.2]{DP10}, and related literature. 

\begin{conj} \label{conj200.01} Assuming the RH and the LI conjecture, the suprema are
\begin{equation}
\lim \inf_{x \to \infty} \frac{\psi(x)-x}{\sqrt{x} (\log \log x)^2}=\frac{-1}{\pi} \qquad \text{and} \qquad \lim \sup_{x \to \infty} \frac{\psi(x)-x}{\sqrt{x} (\log \log x)^2}=\frac{1}{\pi}.
\end{equation}
\end{conj}

More details on the Linear Independence conjecture appear in \cite{IA42}, \cite[Theorem 6.4]{EE85}, and recent literature. The LI conjecture asserts that the imaginary parts of the nontrivial zeros $\rho_n=1/2+i \gamma_n$ of the zeta function $\zeta(s)$ are linearly independent over the set $\{-1,0, 1\}$. In short, the equations  
\begin{equation}
\sum_{1 \leq n \leq M} r_n \gamma_n=0,
\end{equation}
where $r_n \in \{-1,0, 1\}$, have no nontrivial solutions. 

\begin{thm} \label{thm200.02} {\normalfont (Prime number theorem)} Let $x \geq 1$ be a large number. Then
\begin{enumerate} [font=\normalfont, label=(\roman*)]
\item  \text{Unconditionally,}
$$
\pi(x)=\li(x) +O\left (xe^{-c_0 \sqrt{\log x}}\right ) .
$$
\item \text{Unconditional oscillation,}
$$
\pi(x)=\li(x) +\Omega_{\pm} \left (\frac{ x^{1/2}\log \log \log x}{\log x} \right ) . $$

\item \text{Conditional on the RH,}
$$
\pi(x)=\li(x) +O\left (x^{1/2} \log x \right ).  
$$
\end{enumerate}
\end{thm}

\begin{proof} (i) The unconditional part of the prime counting formula arises from the delaVallee Poussin form $\pi(x)=\li(x)+O\left (xe^{-c_0 \sqrt{\log x}}\right )$ of the prime number  theorem, see \cite[p.\ 179]{MV07}. Recent information on the constant $c_0>0$ and the sharper estimate $\pi(x)=\li(x)+O\left ( x e^{-c_0 \log x^{3/5}(\log \log x)^{-2/5}}\right )$ appears in \cite{FK02} and \cite[p.\ 307]{IV03}. The constant $c = .2018$ is computed in \cite{FK02}. \\

(ii) The unconditional oscillations part arises 
from the Littlewood form $\pi(x)=\li(x)+\Omega_{\pm} \left (x^{1/2}\log \log \log  x /\log x \right )$ of the prime number theorem, 
consult \cite[p.\ 51]{IA03}, \cite[p.\ 479]{MV07}, et cetera. \\

(iii)  The conditional part arises from the Riemann form $\pi(x)=\li(x)+O\left (x^{1/2}\log x \right )$ of the prime number theorem. In \cite[Corollary 1]{SL76} there is an explicit version. 
\end{proof}

New explict estimates for the number of primes in arithmetic progressions are computed in \cite{BR18}.

\begin{thm} {\normalfont(\cite{MV75})}  \label{thm200.32} For all real numbers $x > 1$, and and any monotonically increasing function $\theta(x) \geq 2$, 
\begin{equation}
\pi(x+\theta(x))-\pi(x)\leq  \frac{2\theta(x)}{\log \theta(x)} .
\end{equation}
\end{thm}

\begin{thm} {\normalfont(\cite{MH85})}  \label{thm200.58} Let $\theta(x) =(\log x)^r$, where $r>1$. Then 
\begin{equation}
\liminf_{x\to \infty} \frac{\pi(x+\theta(x))-\pi(x)}{\theta(x)/\log x}<1  \qquad \text{ and } \qquad \limsup_{x\to \infty} \frac{\pi(x+\theta(x))-\pi(x)}{\theta(x)/\log x}>1.
\end{equation}
For the range $1<r< e^{\gamma}$, the limit supremum is
\begin{equation}
 \limsup_{x\to \infty} \frac{\pi(x+\theta(x))-\pi(x)}{\theta(x)/\log x}\geq \frac{ e^{\gamma}}{r},
\end{equation}
where $\gamma$ denotes Euler constant.
\end{thm}

\end{document}